\documentclass[11 pt,letterpaper] {amsart}
\usepackage {amssymb,latexsym,amsthm,amsmath,amsfonts, mathtools, multirow, longtable,comment}
\usepackage{enumitem,color}
\usepackage[hidelinks]{hyperref}
\renewcommand{\baselinestretch}{1}
\usepackage{hyperref}
\hypersetup{
	colorlinks,
	citecolor=blue,
	filecolor=black,
	linkcolor=red,
	urlcolor=black
}

\usepackage{mathdots}

\long\def\symbolfootnote[#1]#2{\begingroup
	\def\thefootnote{\fnsymbol{footnote}}\footnote[#1]{#2}\endgroup}
\newcommand{\Z}{\mathbb Z}
\newcommand{\Q}{\mathbb Q}
\newcommand{\R}{\mathbb R}
\newcommand{\A}{\mathbb A}
\newcommand{\tr}{\ensuremath{{}^t\!}}
\newcommand{\tra}{\ensuremath{{}^t}}
\newcommand{\n}{\mathfrak n}
\newcommand{\x}{\mathfrak X}
\newcommand{\h}{\mathfrak h}
\newcommand{\F}{\mathcal F}
\newcommand{\Aut}{\textup{Aut}}
\newcommand{\diag}{\textup{diag}}
\newcommand{\gal}{\textup{Gal}}
\newcommand{\cross}{\times}
\newcommand{\tensor}{\bigotimes}

\newcommand{\sym}{\mathrm{sym}}
\newcommand{\ind}{\mathrm{Ind}}

\DeclareMathOperator{\lcm}{lcm}
\DeclareMathOperator{\ex}{exp}
\makeatletter
\def\imod#1{\allowbreak\mkern10mu({\operator@font mod}\,\,#1)}
\makeatother
\newtheorem{theorem}{Theorem}
\newtheorem{lemma}[theorem]{Lemma}

\newtheorem{proposition}[theorem]{Proposition}
\newtheorem*{theorem*}{Theorem}
\theoremstyle{definition}
\newtheorem{remark}[theorem]{Remark}

\numberwithin{equation}{section}
\newtheorem{notation}[theorem]{Notation}
\newtheorem{defn}[subsubsection]{\bf Definition}

\title[Symmetric power transfers and cuspidal cohomology of ${\rm GL_n}$]{Contribution of symmetric power transfers to the cuspidal cohomology of ${\rm GL_n}$}

\author{Tathagata Mandal and Sudipa Mondal}
\address{Tathagata Mandal, math.tathagata@gmail.com, ISI Kolkata, Kolkata 700108, India}
\address{Sudipa Mondal, sudipa.mondal123@gmail.com, HRI, Prayagraj 211019, India}

\begin{document}
	\begin{abstract}
		Let $\pi$ be a cuspidal automorphic representation of ${\rm GL}_2(\A_\Q)$. Newton and Thorne have proved that for every $n\geq 1$, the symmetric power lifting ${{\rm sym}^n(\pi)}$ is automorphic if $\pi$ is attached to a non-CM Hecke eigenform. In this article, we establish an asymptotic estimate of the number of cuspidal automorphic representations of ${\rm GL}_{n+1}(\A_\Q)$ which contribute to the cuspidal cohomology of ${\rm GL}_{n+1}$ and are obtained by symmetric $n$th transfer of cuspidal representations of ${\rm GL}_2(\A_\Q)$. Here we fix the weight and vary the level. This generalises the previous works done for ${\rm GL}_3$ and ${\rm GL}_4$.
	\end{abstract}
	
	\subjclass[2010]{Primary: }
	\maketitle
	
	\section{Introduction}\label{introduction}
	Automorphic forms and automorphic representations of a reductive algebraic group $G$ are well studied object in number theory. The cohomological approach to this theory was first initiated by Shimura in 1968. It has several applications such as using the cohomology theory, Ash and Steven \cite{ashandsteven} have proven some congruences between Hecke eigenvalues associated to automorphic forms of $G$. Borel studied the stable cohomology of arithmetic groups which has applications in $K$ and $L$-theory \cite[\S 12]{borel}.

	\medskip
	Let $G$ be a connected reductive algebraic group over $\Q$. There is a sum decomposition of the cohomology of an arithmetic subgroup $\Gamma \subseteq G$ into the cuspidal cohomology and Eisenstein cohomology. G. Harder \cite{harder}, J. Schwermer \cite{schwermer} and several authors investigated the Eisenstein cohomology from various point of view. In this article, we focus on the cuspidal cohomology represented by the cuspidal automorphic forms of $G$. A cuspidal automorphic representation $\pi$ of $G(\A_\Q)$ is called cohomological if there exists an irreducible finite dimensional complex representation $(\mathcal{M},V)$ of $G(\R)$ such that $\pi_\infty \otimes \mathcal{M}$ has non-trivial Lie algebra cohomology. It is due to Eichler and Shimura that the cohomology in the case of $G= {\rm GL}_2$ boils down to the theory of modular forms. 
    
    \medskip 
    One of the interesting question in this area is estimating the dimension of cuspidal cohomology of $G= {\rm GL}_n$. There are results from Calegari-Emerton \cite{ce} and Marshall \cite{Marshall} in this direction. Calegari and Emerton have obtained upper bounds for the multiplicities of cohomological representations of semisimple real Lie groups using the noncommutative Iwasawa theory to certain cohomology of locally symmetric spaces. Marshall has studied the growth of cohomology with respect to growing weight of automorphic representations of ${\rm GL}_2$ over number fields that are not totally real. Recently, the same has been studied for ${\rm GL}_3$ in \cite{Marshall2}. Although it is quite a difficult problem to estimate the dimension of the space of automorhpic forms for a general $n$. The other way in which one can get information about the dimension by studying the cuspidal cohomology of ${\rm GL}_n$ obtained from lower rank groups ${\rm GL}_{n'}$ using Langlands functoriality.

	\medskip 
	
	In this article, we aim to study the later part. Let $\pi$ be a cuspidal auromorphic representation of ${\rm GL}_2(\A_\Q)$. Due to the recent work of Newton and Thorne \cite{newtonthorne}, ${\rm sym}^n(\pi)$ is an automorphic representation of ${\rm GL}_{n+1}(\A_\Q)$. Hence it is natural to ask how much of cuspidal cohomology of ${\rm GL}_{n+1}$ is captured by the symmetric $n$th transfer of cuspidal automorphic representations of ${\rm GL}_2$.
	
	\medskip 
	
	The cases of $n=2$ and $n=3$ were studied in \cite{ambi} by Ambi and in \cite{chandrasheel-sudipa} by Bhagwat and Mondal respectively. More precisely, they have estimated the number of cohomological cuspidal automorphic representations of ${\rm GL}_3$ and ${\rm GL}_4$ obtained from ${\rm GL}_2$  by ${\rm sym}^2$ and ${\rm sym}^3$ transfer respectively corresponding to a specific level structure. Our aim is to generalise these results for any $n\geq 2$.
	
	\medskip 
	
	Let $\mathcal{I}_{k,n}(t)$ denote the set of all cohomological cuspidal automorphic representations $\Pi$ of ${\mathrm{GL}}_{n+1}(\A_\Q)$ corresponding to the highest weight $\mu_k$ and level $K_f^{n+1}(t)$ (see \S \ref{cohomology of gln}, \S \ref{level-structure}) such that $\Pi={\mathrm{sym}^n(\pi)}$ for a cuspidal automorphic representation $\pi$ of highest weight $\big(\frac{k}{2}-1, 1-\frac{k}{2}\big)$ and level structure $K_f^{2}(t')$($t' \in \mathbb{N}$) of ${\mathrm{GL}}_{2}(\A_\Q)$.
	
	\begin{theorem}
		Let $k,n \in \mathbb{N}$ with $k\geq2$ even, n$\geq 2$. We have
		$$ p^{\frac{2}{n+2}j} \ll_k |\mathcal{I}_{k,n}(p^j)| ~~~~~~~~~~ \text{ as the prime } p \to \infty$$
		where the implied constant depends upon $k$.
	\end{theorem}
	
	We mention that our theorem improves the results obtained in \cite{ambi}, \cite{chandrasheel-sudipa}, see Remark \ref{improve}. To estimate the above number, we use the dimension formulae for spaces of cusp forms and of CM forms which is discussed in section \ref{preliminaries}. We have managed to establish a relation between local conductors of a local representation and its $n$th symmetric power in section \ref{relation between conductors}. Finally in section \ref{cohomology of gln}, we discuss the connection between a cohomological representation and symmetric $n$th transfer of an automorphic representation and prove our main theorem using all the results developed in our previous sections. \\
	
	\section{Preliminaries}\label{preliminaries}
	\subsection{Symmetric $n$th transfer :} Let $F$ be a number field and $\pi= \bigotimes \limits_{v} \pi_v$ be a cuspidal automorphic representation of ${\mathrm{GL}}_2(\A_F)$ where $\mathbb{A}_F$ is the ring of ad\`eles over $F$. For each place
	$v$, let $\phi_v$ be the two dimensional representation of the Weil-Deligne group attached to $\pi_v$. Consider the $n$th symmetric power ${\mathrm{sym}}^n : {\mathrm{GL}_2} \to {\mathrm{GL}_{n+1}}$ of the standard representation of ${\mathrm{GL}_2}$. Then ${\mathrm{sym}}^n(\phi_v)= {\mathrm{sym}}^n \circ \phi_v$ is an $n+1$ dimensional representation of the Weil-Deligne group. Let ${\mathrm{sym}}^n(\pi_v)$ be the irreducible admissible representation of ${\mathrm{GL}}_{n+1}(F_v)$ attached to ${\mathrm{sym}}^n(\phi_v)$ by the local Langlands corresspondence (see \cite{henniart},\cite{haristaylor},\cite{knapp},\cite{kudla}). Set ${\mathrm{sym}}^n(\pi):= \bigotimes \limits_{v} {\mathrm{sym}}^n(\pi_v)$ which is known as the symmetric $n$th transfer of $\pi$.
	
	We only consider the case when $F=\Q$, and in this case, the automorphy of ${\mathrm{sym}}^n(\pi)$ is proved in \cite{newtonthorne} which we mention in section \ref{cohomology of gln}. 
	
	Let $N \in \mathbb{N} $. Consider the following congruence subgroups of $\text{SL}_2(\mathbb{Z})$
	\begin{align*}
		\Gamma_1(N) & = \left\{A \in \text{SL}_2(\mathbb{Z}): A \equiv \begin{bmatrix}
			1 & * \\
			0 & 1 
		\end{bmatrix}(\text{mod } N) \right\}, 
	\end{align*}
	and
	\begin{align*}
		\Gamma_0(N) & = \left\{A \in \text{SL}_2(\mathbb{Z}): A \equiv \begin{bmatrix}
			* & * \\
			0 & * 
		\end{bmatrix}(\text{mod } N) \right\}.
	\end{align*}
	
	
	Let $S_k(\Gamma_1(N))$ be the space of cusp forms of weight $k \in \mathbb{N}$ for the congruence subgroup $\Gamma_1(N)$ and we denote $S_k(\Gamma_0(N),\epsilon)$ to be the space of cusp forms with Dirichlet character $\epsilon$. We have, $S_k(\Gamma_1(N))= \bigoplus_{\epsilon} S_k(\Gamma_0(N),\epsilon)$. 
	
	We have the following dimensions formulae for $S_k(\Gamma_1(N))$ and $S_k^{\text{new}}(\Gamma_1(N))$, the space of newforms of weight $k$ and level $N$. 
	\begin{lemma}\label{dimensioncuspforms}
		We have,
		$$ \frac{\rm{dim}_{\mathbb{C}}(S_k(\Gamma_1(N)))}{N^2}= \frac{k-1}{4\pi^2}+o(1) \text{ as } N \longrightarrow \infty.$$
	\end{lemma}
	
	\begin{lemma}\label{dimensionnewforms} 
		For $i \geq 2$, we have
		$$  \frac{{\rm dim}_\mathbb{C}~(S_k^\text{new}(\Gamma_1(p^i)))}{p^{2i}} = \left( \frac{k-1}{4\pi^2} \right) \left(1-\frac{1}{p^2}\right)^2 + o(1) \quad \text{ as the prime } ~p \longrightarrow \infty.$$
	\end{lemma}
	
	For more details about the dimension formulae, see \cite{ambi},\cite{chandrasheel-sudipa2},\cite{stein}. One can construct cusp forms from Hecke character $\zeta$ of a quadratic number field $K$. This construction is known as automorphic induction. 
	\begin{theorem}\cite[Th. 12.5]{iwaniec} Let $\zeta \pmod{\mathfrak{m}}$ be a Hecke character of $K=\Q(\sqrt{-D}), D>0$ with $\infty$-type $k>1$. Then 
		$$f(z)= \sum_{\mathfrak{a}} \zeta(\mathfrak{a}){N_{K/\Q}(\mathfrak{a})}^{\frac{k}{2}}e(zN_{K/\Q}(\mathfrak{a})) \in S_{k+1}(\Gamma_0(N),\chi)$$
		where $N=D\cdot N_{K/\Q}(\mathfrak{m})$ and $\chi\pmod{N}$ is the Dirichlet character given by $\chi(i)=\chi_D(i)\zeta((i))$ for $i\in \mathbb{N}$. Here, $\chi_D$ is the Kronecker symbol.
	\end{theorem}
	
	The cusp form obtained in this way is called a CM cusp form. From the work of Ambi \cite{ambi}, we can count the number of CM cusp forms of a given weight and level.
	
	\begin{notation}\label{ckn}
		Let $C_k(N)$ denote the set of normalised cusp eigenforms of Hecke operators of $\Gamma_1(N)$ of weight $k$ obtained by automorphic induction of Hecke characters (gr\"o{\ss}encharacters) of all possible imaginary quadratic extensions of $\Q$.
	\end{notation}

	\begin{notation}
		Let $f$ and $g$ be two non-negative real valued functions on $\mathbb{N}$. Suppose there exists a constant $C$ and $n_0\in \mathbb{N}$ such that 
		$$f(x) \leq Cg(x) \quad \forall ~ x\geq n_0.$$ 
		Then we say $f \ll g$ (or equivalently $g \gg f$). If both $f \ll g$ and $g \ll f$ hold, we write $f \sim g$.
	\end{notation}
	
	Define $\hat{N}:=\prod_{p|N}p$. Then we have the following upper bound of $|C_k(N)|$ which plays an important role in proving our main theorm.
	\begin{theorem}\cite[Th. 1.1]{ambi}
		Let $k\geq 2$ and $N\in \mathbb{N}$. Then,
		\[
		|C_k(N)| \ll_{k,\epsilon} N\cdot \hat{N}^{1+\epsilon} ~~~~~ \text{  } ~~~~~~~ \forall ~\epsilon \in (0,1)
		\]
		as $N \to \infty$ where the implied constant depends upon $k,\epsilon$.
	\end{theorem}
	
	The authors in \cite{chandrasheel-sudipa2} have improved the above theorem when $N=p^n$ for a prime $p$.
	
	\begin{proposition}\label{improvment}
		Let $k\geq 2$ and $i\in \mathbb{N}$. We have,
		$$|C_k(p^i)|\ll_{\epsilon} p^{i+\frac{1}{2}+\epsilon} ~~~~~~~\text{ as } \epsilon \to 0, p \to \infty.$$
	\end{proposition}

	\subsection{The Level Structure :}  \label{level-structure}
	Consider the algebraic group ${\rm GL}_m$. For a finite prime $p$ of $\Q$, let $\Z_p$ be the ring of $p$-adic integers. For each integer $r \geq 0$, define
	$$ K_p^m(r):= \{ x=(x_{i,j})_{m\times m} \in {\rm GL}_m(\mathbb{Z}_p) : x_{m,k}\in p^r\mathbb{Z}_p, 1\leq k<m, ~x_{m,m}-1 \in p^r\Z_p \}.$$
	
	Let $\mathbb{A}_f$ denote the finite part of ad{\`e}les over $\Q$. Let $N$ be a positive integer with prime factorization $N =  \prod \limits_{1 \leq i \leq s, p_i | N}p_i^{n_i}$. Define a compact open subgroup $K_f^m(N) = \prod \limits_p K_p$ of ${\rm GL}_m(\mathbb{A}_f)$ where 
	$$K_p =
	\begin{cases} 
		& K_{p_i}^m{(n_i)} \quad \text{if} ~ p \mid N~\text{i.e.},~ p = p_i  \\
		& {\rm GL}_m(\mathbb{Z}_p) \quad \text{if}~ p \nmid N.
	\end{cases}
	$$
	For each $N \ge 1$, the compact open subgroup $K_f^m(N) \subseteq {\rm GL}_m (\mathbb A_f)$ is called the level structure corresponding to $N$.

	\section{Relation between conductors}\label{relation between conductors}
	\subsection{Conductor of characters and representations :}
	Let $K$ be a non-archimedean local field of characteristic zero. Let $\mathcal{O}_K, \mathfrak{p}_K, \kappa_{K}$ be the ring of integers, the maximal ideal and the residue field of $K$ respectively.	
	\begin{defn}
		Let $\chi$ be a multiplicative character of $K^\times$. The conductor $c(\chi)$ of $\chi$ is the smallest positive integer $i$ such that $\chi_{|_{1+{\mathfrak{p}_K^i}}}=1$. For two multiplicative characters $\chi_1, \chi_2$  of $K^\times$, we have $c(\chi_1 \chi_2) \leq \text{max}\{c(\chi_1), c(\chi_2)\}$.
	\end{defn}
	
	Let $(\rho, \mathcal{H})=  (\bigotimes_{p\leq \infty}\rho_p, \bigotimes_{p\leq \infty} \mathcal{H}_p )$ be an irreducible automorphic representation of ${\rm GL}_m (\mathbb A_\Q)$. For each finite prime $p$, consider the space $$\mathcal{H}_p^{K_p^m(r)}:= \{v \in \mathcal{H}_p:x\cdot v=v \text{ } ~~~\forall v \in K_p^m(r) \}.$$
	\begin{defn}
		The conductor of $\rho_p$ is the smallest positive integer $c(\rho_p)$ such that the space $\mathcal{H}_p^{K_p^m(c(\rho_p))}$ is non-zero. The conductor of $\rho$ is defined as $N_\rho=\prod_{p<\infty}p^{c(\rho_p)}$.
	\end{defn}
	Note that, $N_\rho$ is finite since $c(\rho_p)=0$ for almost all prime $p$. \\
	
	Let $\epsilon$ be a Dirichlet character modulo $N \in \mathbb{N}$ and $f\in S_k(\Gamma_0(N),\epsilon)$. Write $N=p^{N_p}N'$ with $(p,N')=1$ and $\epsilon= \epsilon_p\cdot \epsilon'$ where $\epsilon_p$ is the $p$-part of $\epsilon$ and $\epsilon'$ is the prime to $p$-part. We have $c(\epsilon_p)\leq N_p$.  Let $\omega$ be the adelization of $\epsilon$. Then the central character of the cuspidal representaion $\pi_{f}$ attached to $f$ is $\omega$. Let $\omega_p$ be the restriction of $\omega$ to $\Q_p^\times$. For $x \in \Q_p^\times$, let $[x]= (1,\cdots, 1, x, 1,\cdots)$ denotes the corresponding element of $x$ in $\A_\Q^\times$ ($x$ is in $p$-th position). Write $x=p^iu$ where $i\in \Z$ and $u\in \mathbb{Z}_p^\times$. Then $\omega_p$ is given by the following formula: 
	\begin{equation}\label{adelization}
		\omega_p([p^iu])= \epsilon'(p)^i \epsilon_p^{-1}(u).
	\end{equation}
    This shows that the restriction of $\omega_p$ to $\Z_p^\times$ is the inverse of the character  $\epsilon_p$ on $\Z_p^\times$ and  $c(\omega_p)=c(\epsilon_p)$ \cite[\S 2]{loeffler}.

	Now we write down the local parameter of $\pi_p:= \pi_{f,p}$ , the local representation attached to $f$ at $p$ and ${\mathrm{sym}}^n(\pi_p)$ to calculate the conductor relation of these representations.

	\subsection{Local Parameter of $\pi_p$ :}\label{localparametrofpip} Suppose the residue field $\kappa_K$ is of characteristic $p > 0$ and let $q = |\kappa_K|$. Let $W(K)$ be the Weil group of $K$. Consider a finite dimensional vector space $V$ and a continuous homomorphism i.e. a representation of $W(K)$ given by $$\phi : W(K) \rightarrow {\rm GL}(V).$$
	
	Let $N$ be a nilpotent endomorphism of $V$ satisfying
	$$\phi(g)N\phi(g)^{-1}= w(g)N \quad \forall ~g \in W(K),$$ where $w$ is the unramified quasicharacter giving the action of $W(K)$ on the roots of unity, then the pair $(\phi,N)$ is said to be Weil-Deligne representation. See \cite{tate} for more details. We will mostly consider the case when $K=\Q_p$.
	
	By the local Langlands correspondence, a two dimensional Weil-Deligne representation $(\phi,N)$ corresponds uniquely to an irreducible admissible representation $\pi_p$ of ${\rm GL}_2(\Q_p)$ and vice-versa. Furthermore, the equivalence preserves $L$-functions, $\epsilon$- factors and conductors (see \cite{kudla} and \cite{kutzko}). We refer to $(\phi,N)$ as the local parameter (or the $L$-parameter) of $\pi_p$.

	\begin{enumerate}

		\item Let $\pi_p=\pi(\mu_1,\mu_2)$ be a principal series representation where $\mu_1\mu_2^{-1}=|\cdot|^{\pm1}$. Then the $L$-parameter of $\pi_p$ is given by 
		\begin{equation}
			\phi(x)= \begin{bmatrix}
				\mu_1(x) & \\ & \mu_2(x)	\end{bmatrix} , \, x\in W(\Q_p) \text{ and } N=0.
		\end{equation}
		
		\item Let $\pi_p= \mu\otimes {\rm St}_{{\rm GL}_2}$ be a special representation where $\mu$ is a character of $K^\times$. The $L$-parameter of $\pi_p$ is given by \begin{equation}
			\phi(x)= \begin{bmatrix}
				\mu(x)|x|^{\frac{1}{2}} & \\ & \mu(x)|x|^{-\frac{1}{2}}	\end{bmatrix} , \, x\in W(\Q_p) \text{ and } N=\begin{bmatrix}
				0 & 1 \\ 0 & 0
			\end{bmatrix}.
		\end{equation}
		
		\item The local parameter of a supercuspidal representation is an irreducible representation of $W(\Q_p)$ with $N=0$.
	\end{enumerate}

	\subsection{Local parameter of $\sym^n(\pi_p)$ and conductor relations :} The $n$th symmetric power of the standard representation of ${\rm GL}_2(\mathbb{C})$ is equivalent to the action of ${\rm GL}_2(\mathbb{C})$ on the vector space of homogeneous polynomials of degree $n$. We use its matrix representation to calculate the $L$-parameter of ${\rm sym}^n(\pi_p)$.
	
	Suppose the $L$-parameter of ${\rm sym}^n(\pi_p)$ is denoted by $({\rm sym}^n(\phi), N')$. For a principal series representation $\pi_p$, we have
	\begin{equation} \label{prin}
		{\rm sym}^n(\phi)(x)=
		\begin{bmatrix}
			\mu_1^n(x) & & & \\ & \mu_1^{n-1}(x)\mu_2(x) & &\\ & & \ddots & \\ & & & \mu_2^n(x)
		\end{bmatrix} \, x\in W(\Q_p) \text{ and } N'=0.
	\end{equation}
	
	
	If $\pi_p$ is a special representation of ${\rm GL}_2(\Q_p)$, then we have	\begin{equation} \label{special}
		{\rm sym}^n(\phi)(x)=
		\begin{bmatrix}
			\mu^n(x)|x|^{\frac{n}{2}} & & & \\ & & \ddots & \\ & & & 	\mu^n(x)|x|^{-\frac{n}{2}}
		\end{bmatrix} \, x\in W(\Q_p) \text{ and } N'=\begin{bmatrix}
			& & & 1 \\ & & 0 &\\ & \iddots & & \\ 0 & & &
		\end{bmatrix}.
	\end{equation}
	
	From \eqref{prin} and \eqref{special}, we can easily deduce the conductors of ${\rm sym}^n(\pi_{p})$ when $\pi_{p}$ is either principal series representation or special representation. With the notations as above, we have the following.
	\begin{lemma} \label{lem1}
		\begin{enumerate}
			\item 
			If $\pi_p$ is a principal series representation, then 
			\[c({\rm sym}^n(\phi)) 
			\leq n\cdot c(\phi).\]
			\item 
			If $\pi_p$ is a special representation, then 
			\[c({\rm sym}^n(\phi))= \begin{cases}
			n & \text{ if } \mu \text{ unramified } \\
			(n+1)\cdot c(\mu^n) & \text{ if } \mu \text{ ramified. }
			\end{cases} \]
		\end{enumerate}
	\end{lemma}
	\begin{proof}
		If $\pi_p$ is a principal series representation, then from Equ. \ref{prin} we have that 
		\begin{equation} \label{prin1}
			c({\rm sym}^n(\phi))= \sum_{\substack{i+j=n\\ i,j \in \{ 0,\cdots, n\}}}c(\mu_1^i\mu_2^j).
		\end{equation} Note that $c(\phi)=c(\mu_1)+c(\mu_2)$. Therefore, we deduce $c(\mu_1^i\mu_2^j) \leq c(\phi)$ and hence from \eqref{prin1} it follows that $c({\rm sym}^n(\phi)) \leq n \cdot c(\phi)$. If $\pi_p$ is a special representation, using \cite[Section 10, Prop.]{rohrlich} and Equ. \ref{special} we obtain the desired result.
	\end{proof}
	
	Now, we assume that $\pi_p$ is a supercuspidal representation of ${\rm GL}_2(\Q_p)$ with $p$ odd. If $(\phi,0)$ is the local parameter of $\pi_p$, then $\phi= {\rm Ind}_{W(K)}^{W(\Q_p)}(\eta)$ where $K$ is a quadratic extension of $\Q_p$ with $\eta \neq \eta^\sigma, \sigma \in W(\Q_p)\backslash W(K)$ and $\eta^\sigma(x)=\eta(\sigma x \sigma^{-1}), x\in W(K)$. If the conductor of $\eta$ is $c(\eta)$, then the conductor $c(\phi)$ of $\phi$ is given by
	\begin{equation} \label{indconductor}
		c(\phi)= {\rm dim}({\eta})~v(d_{K/\Q_p}) + f_{K/\Q_p}~ c(\eta),
	\end{equation}
	where $d_{K/\Q_p}$ is the discriminant of the field extension $K/\Q_p$ and $f_{K/\Q_p}$ is the residual degree (see \cite{serre} for more details). 
	
	With a suitable basis, $\phi$ has following matrix form: 
	\begin{equation}
		\phi(x)= \begin{bmatrix}
			\eta(x) & \\ & \eta^\sigma(x)	\end{bmatrix} , \, x\in W(K) \text{ and } \phi(\sigma)= \begin{bmatrix}
			& 1 \\ \eta(\sigma^2) & 
		\end{bmatrix}.
	\end{equation}
	Then,
	\begin{equation}
		{\rm sym}^n(\phi)(x)= \begin{bmatrix}
			\eta^n(x) & & & \\ & \eta^{n-1}(x)\eta^\sigma(x) & & \\ & & \ddots &\\ & & & (\eta^\sigma)^n(x)
		\end{bmatrix} \text{ and } 
	\end{equation} 
	\begin{equation}
		{\rm sym}^n(\phi)(\sigma)= \begin{bmatrix}
			& & & 1 \\ & & \eta^\sigma(\sigma^2) & \\ & \iddots & & \\ (\eta^\sigma)^n(\sigma^2)		\end{bmatrix}.
	\end{equation}
	From the above two equations, we obtain that
	\begin{proposition}\label{symndirectsum}
		If $\pi_p$ is a supercuspidal representation with local parameter $(\phi,0)$, then $\sym^n(\phi)$ is given as follows.
		\begin{enumerate}
			\item 
			If $n$ is odd, then ${\rm sym}^n(\phi)$ is a sum of $\frac{n+1}{2}$ number of two dimensional representations, i.e.
			\begin{equation}\label{n odd}
			{\rm sym}^n(\phi)= \bigoplus_{i=0}^{{(n-1)}/2}{\rm Ind}_{W(K)}^{W(\Q_p)}(\eta^{n-i}(\eta^\sigma)^i).
			\end{equation}
			\item 
			If $n$ is even, then 
			\begin{equation}\label{n even}
			{\rm sym}^n(\phi)= \bigoplus_{i=0}^{n/2-1}{\rm Ind}_{W(K)}^{W(\Q_p)}(\eta^{n-i}(\eta^\sigma)^i) \oplus \eta^{\frac{n}{2}}(\eta^{\sigma})^{\frac{n}{2}}.
			\end{equation}
		\end{enumerate}
	\end{proposition}


	
    Let $f\in S_k(\Gamma_0(N),\epsilon)$ and $p \mid N$ be an odd prime such that $\pi_p := \pi_{f,p}$ is a supercuspidal representation, that is, if $(\phi,0)$ is the local parametr of $\pi_p$, then $\phi = {\rm Ind}_{W(K)}^{W(\Q_p)}(\eta)$ with $K/\Q_p$ quadratic. Then using Equ. \ref{indconductor},  we have
	\begin{equation}\label{supconductor}
	c(\pi_p)= N_p = \begin{cases}
		2c(\eta) & \text{ if } K/\Q_p \text{ is unramified }\\
		1+c(\eta) & \text{ if } K/\Q_p \text{ is ramified. }
	\end{cases} 
    \end{equation} 
	The central character of $\pi_{p}$ is $\eta|_{\Q_p^\times} \cdot \omega_{K/\Q_p}$ where $\omega_{K/\Q_p}$ is the quadratic character of $\Q_p^\times$ associated to the quadratic extension $K/\Q_p$ such that it is trivial on the norm group $N_{K/\Q_p}(K^\times)$. 
	
	Recall, $\omega_p$ be the $p$-part of the central character $\omega$ of $\pi_{f}$. From now on we identify $\omega_p$ as a character of $\Q_p^\times$ since $\Q_p^\times$ is isomorphic to a subgroup of $\A_\Q^\times$ under the map $x \mapsto [x]$ which is discussed before \S \ref{localparametrofpip}. Hence we have $\eta|_{\Q_p^\times} \cdot \omega_{K/\Q_p}=\omega_p$ on $\Q_p^\times$. Now, evaluating at $N_{K/\Q_p}(x)$ for $x \in K^\times$, we deduce that $\eta^\sigma=\eta^{-1}\cdot \omega_p\circ N_{K/\Q_p}  \ \text{ on } K^\times$. Writing $\omega_p'= \omega_p \circ N_{K/\mathbb{Q}_p}$, we have  
	\begin{equation} \label{central} 
		\eta^\sigma=\eta^{-1} \omega_p' \ \text{ on } K^\times.
	\end{equation}

	Next, we determine the situation when the representations involved in the direct sums of Prop. \ref{symndirectsum} will be irreducible and isomorphic.
	\begin{lemma}
		Let $i,j,k,l\geq 0$ with $i+j=k+l=n$. Then we have the following:
		\begin{enumerate}
			\item ${\rm Ind}_{W(K)}^{W(\Q_p)}(\eta^{i}(\eta^\sigma)^j)$ is irreducible if and only if $\eta^{2(i-j)}\neq \omega_p'^{(i-j)}$ on $K^\times$. 
			
			\item ${\rm Ind}_{W(K)}^{W(\Q_p)}(\eta^{i}(\eta^\sigma)^j) \cong {\rm Ind}_{W(K)}^{W(\Q_p)}(\eta^{k}(\eta^\sigma)^l)$ if either $\eta^{2(i+k-n)} = (\omega_p')^{i+k-n}$ on  $K^\times$ or $\eta^{2(i-k)} \linebreak = (\omega_p')^{i-k}$ on $K^\times$.
		\end{enumerate}
		
	\end{lemma}
	
	\begin{proof}
		\begin{enumerate} 
			\item We prove this by contradiction. Let ${\rm Ind}_{W(K)}^{W(\Q_p)}(\eta^{i}(\eta^\sigma)^j)$ be reducible. This is true if and only if $\eta^{i}(\eta^\sigma)^j= (\eta^{i}(\eta^\sigma)^j)^\sigma$  on $K^\times$. Since $\sigma$ is an involution, we have that $\eta^{i}(\eta^\sigma)^j= (\eta^\sigma)^i\eta^{j}$  on $K^\times$. Now using Equ. \ref{central}, we get  $\eta^{2(i-j)}=\omega_p'^{(i-j)}$  on $K^\times$. This proves the first part of the lemma. 
			
			\item Next, assume that ${\rm Ind}_{W(K)}^{W(\Q_p)}(\eta^{i}(\eta^\sigma)^j)$ is isomorphic to ${\rm Ind}_{W(K)}^{W(\Q_p)}(\eta^{k}(\eta^\sigma)^l)$. This holds true if and only if either $\eta^{i}(\eta^\sigma)^j = \eta^{k}(\eta^\sigma)^l$ or $\eta^{i}(\eta^\sigma)^j= (\eta^{k}(\eta^\sigma)^l)^\sigma$ on $K^\times$. By Equ. \ref{central}, these conditions reduce to $\eta^{2(i+k-n)} = (\omega_p')^{i+k-n}$ and $\eta^{2(i-k)} = (\omega_p')^{i-k}$ on $K^\times$ which completes the proof.
		\end{enumerate}
	\end{proof}
	\begin{remark}
	As $\ind_{W(K)}^{W(\Q_p)}(\eta)$ is irreducible, we have $\eta \neq \eta^\sigma$ and by Equ. \ref{central}, $\eta^2 \neq \omega_p'$ on $K^\times$. Thus, using the above lemma, we conclude that when $n$ is odd, ${\rm Ind}_{W(K)}^{W(\Q_p)}(\eta^{\frac{n+1}{2}}(\eta^\sigma)^{\frac{n-1}{2}})$ is always irreducible. Depending upon whether $\ind_{W(K)}^{W(\Q_p)}(\eta^{i}(\eta^\sigma)^j)$ is irreducible
	or reducible, we can classify the types of ${\rm sym}^n(\pi_{p})$ of a supercuspidal representation $\pi_{p}$. These types are determined using the variation of epsilon factors in \cite{bmm} when $n=3$.
	\end{remark}
	
	Next we prove a relation between $c({\rm sym}^n(\pi_p))$ and $c(\pi_p)$. Before we do that, let us recall the following useful lemma to compute the conductor of a character obtained by composing with the norm map.

	\begin{lemma}\cite[Lemma $1.8$]{tunnell}\label{conductor norm map}
		Let $K$ be a quadratic separable extension of $F$. Let $\chi$ be a quasicharacter of $F^\times$. Then $f_{K/F} c(\chi\circ N_{K/F})=c(\chi)+c(\chi\omega_{K/F})-c(\omega_{K/F})$, where $f_{K/F}$ denotes the residual degree .
	\end{lemma}

	\begin{lemma} \label{etasigma}
		We have \[ c(\eta^\sigma) \leq \begin{cases}
			c(\pi_{p}) & \text{ if } K/\Q_p \text{ is unramified }\\
			2c(\pi_{p}) & \text{ if } K/\Q_p \text{ is ramified. }
		\end{cases} \]
	\end{lemma}
    \begin{proof}
        From Equ. \ref{central}, it follows that $c(\eta^\sigma) \leq {\mathrm{max}}\{c(\eta),c(\omega_p')\}$. Using Lemma \ref{conductor norm map}, we have $c(\omega_p')=c(\omega_p)$ if $K/\Q_p$ is unramified; and $c(\omega_p)+c(\omega_p\omega_{K/\Q_p})-1$ if $K/\Q_p$ is ramified. As $c(\omega_p)=c(\epsilon_p) \leq N_p$, using Equ.  \ref{supconductor} we complete the proof of the lemma.
    \end{proof}

	
	\begin{proposition}\label{conductor relation}
		For an odd prime $p$, we have the following conductor relation: 
		$$
		1 \leq c({\rm sym}^n(\pi_p))\leq (n+2)c(\pi_p).
		$$
	\end{proposition}
	
	\begin{proof}
		The inequality holds when $\pi_p$ is a ramified principal series or special representation, follows from Lemma \ref{lem1}. So let us assume that $\pi_p$ is a supercuspidal representation. We first show that $c({\rm Ind}_{W(K)}^{W(\Q_p)}(\eta^{i}(\eta^\sigma)^j))\leq 2c(\pi_p)$ for all $i,j \geq 0$ with $i+j=n$. Suppose ${\rm Ind}_{W(K)}^{W(\Q_p)}(\eta^{i}(\eta^\sigma)^j)$ is irreducible. Then by \eqref{indconductor},
		\[ 
		c \big({\rm Ind}_{W(K)}^{W(\Q_p)}(\eta^{i}(\eta^\sigma)^j) \big)
		= \begin{cases}
			2c(\eta^{i}(\eta^\sigma)^j) & \text{ if } K/\Q_p \text{ is unramified }\\
			c(\eta^{i}(\eta^\sigma)^j)+1 & \text{ if } K/\Q_p \text{ is ramified. }
		\end{cases} 
		\]
		Now, if $K/\Q_p$ is unramified, then from Lemma \ref{etasigma}, $c(\eta^\sigma)\leq c(\pi_p)$. Since $c(\eta^{i}(\eta^\sigma)^j) \leq \text{max} \{c(\eta), \linebreak c(\eta^\sigma)\}$, we deduce that
		\begin{eqnarray}\label{conductor inequality}
			c({\rm Ind}_{W(K)}^{W(\Q_p)}(\eta^{i}(\eta^\sigma)^j))\leq 2c(\pi_p).
		\end{eqnarray}
		On the other hand, if $K/\Q_p$ is ramified then using Lemma \ref{etasigma} we have that $c(\eta^\sigma)\leq 2c(\pi_p)$. In fact, $c(\eta^\sigma)+1 \leq 2c(\pi_p)$. By the similar arguments used in the unramified case, we obtain the inequality \ref{conductor inequality}.
		
		Next, assume that ${\rm Ind}_{W(K)}^{W(\Q_p)}(\eta^{i}(\eta^\sigma)^j)$ is reducible. Then ${\rm Ind}_{W(K)}^{W(\Q_p)}(\eta^{i}(\eta^\sigma)^j)= \varphi_{ij}\oplus \varphi_{ij}\omega_{K/\Q_p}$ for some character $\phi_{ij}$ of $\Q_p^\times$, where $\omega_{K/\Q_p}$ is the unique quadratic character attached to $K/\Q_p$. Also,  $\eta^{i}(\eta^\sigma)^j$ factors through the norm map, i.e., $\eta^{i}(\eta^\sigma)^j= \phi_{ij}\circ N_{K/\Q_p}$. By Lemmas \ref{conductor norm map} and \ref{etasigma}, we have $c(\phi_{ij})\leq c(\pi_p)$. Hence, 
		\begin{eqnarray}
			c({\rm Ind}_{W(K)}^{W(\Q_p)}(\eta^{i}(\eta^\sigma)^j))\leq 2c(\pi_p).
		\end{eqnarray}
		Hence from \eqref{n odd} and \eqref{n even}, we deduce that $c({\rm sym}^n(\pi_p))\leq (n+2)c(\pi_p)$.
	\end{proof}
	
	\section{Cuspidal Cohomology of ${\rm GL}_n$ and proof of the main Theorem}\label{cohomology of gln}
	\subsection{Cuspidal cohomology of ${\rm GL}_n$}	Let $\rho$ be a cuspidal automorphic representation of ${\rm GL}_n/\Q$ with $\rho_\infty$ and $\rho_f$ be its finite and infinite parts respectively. Let $\mu$ be a dominant integral weight corresponding to the standard Borel subgroup and $\mathcal M_{\mu,\mathbb{C}}$ be the underlying vector space of the finite dimensional irreducible highest weight representation corresponding to $\mu$. For a given level structure $K_f^n \subset  {\rm GL}_n(\mathbb{A}_f)$, we denote by ${\rho}^{K_f^n}_f$, the space of $K_f^n$-fixed vectors of $\rho_f$. We say that $\rho$ is cohomological if the relative Lie algebra cohomology with coefficients in $\mathcal M_{\mu,\mathbb{C}}$ and level $K_f^n$ is non zero in some degree i.e.,
	$$ H^*(\mathfrak{gl}_n, \mathbb{R}_+^\times \cdot {\rm SO}_n(\mathbb{R}),\rho_\infty \bigotimes \mathcal M_{\mu,\mathbb{C}})\bigotimes \rho_f^{K_f^n} \neq 0 .$$
	
	\noindent We denote this by $\rho \in {\mathrm{Coh}}({\rm GL}_n, \mu, K^n_f)$. For more details, the reader is referred to \cite[Section 2.3]{raghuram}. Note that, the cohomological criterion of $\rho$ depends only on its infinite part $\rho_\infty$.\medskip
	
	Let $f$ be a holomorphic cusp form of even weight $k\geq 2$. If $\pi_f= \bigotimes \limits_{p \leq \infty} \pi_p$ is the cuspidal representation attached to $f$, then  $\pi_\infty$ is a discrete series representation $D_{k-1}$ and the corresponding dominant integral weight is $\lambda_k=\big(\frac{k}{2}-1, 1-\frac{k}{2}\big)$. Also $\pi_f$ is cohomological i.e. $\pi_f\in {\mathrm{Coh}}({\rm GL}_2, \lambda_k, K^2_f)$ \cite[Example 5.2]{raghuramshahidi}. Define
	\begin{equation}\label{muk}
		\mu_k= 
		\begin{cases}
			\big(n(\frac{k}{2}-1), (n-2)(\frac{k}{2}-1), \cdots, (\frac{k}{2}-1), (1-\frac{k}{2}),\cdots, n(1-\frac{k}{2})\big) & \text{ if }  n \text{ is odd }\\
			\big(n(\frac{k}{2}-1), (n-2)(\frac{k}{2}-1), \cdots, 2(\frac{k}{2}-1),0,2(1-\frac{k}{2}),\cdots, n(1-\frac{k}{2})\big) & \text{ if } n \text{ is even. }
		\end{cases}
	\end{equation}
	
	Then, 
	\begin{lemma}
		${\mathrm{sym}^n}(\pi_f) \in {\mathrm{Coh}}({\rm GL}_{n+1}, \mu_k, K^{n+1}_f)$  with respect to the weight $\mu_k$ given in Equ. \ref{muk}.
	\end{lemma}
	
	\begin{proof}
		This follows from Th. 3.2 in \cite{raghuram}.
	\end{proof}
	
	This shows that the symmetric $n$th lift preserves cohomology. From this, one might think that the functorial lift of cohomological representation is cohomological but this is not true in general. For this fact, see the example 5.6 in \cite{raghuramshahidi}. 
	
	The following theorem \cite[Th. A]{newtonthorne} proves the automorphy of ${\mathrm{sym}}^n(\pi)$ of a cuspidal reprersentation $\pi$ of ${\mathrm{GL}}_2$ and provides a cuspidality criterion of it which we state below.
	
	\begin{theorem}\label{cuspidalitysymn}
		Let $\pi$ be a regular algebraic, cuspidal automorphic representation of ${\mathrm{GL}}_{2}(\A_\Q)$. Suppose that $\pi$ is non-CM. Then for each integer $n \geq 1$, ${\mathrm{sym}}^n(\pi)$ exists, as a regular algebraic, cuspidal
		automorphic representation of ${\mathrm{GL}}_{n+1}(\A_\Q)$.
	\end{theorem}
	Using the above theorem we can say that if $f \notin C_k(N)$  (see notation \ref{ckn}), then ${\mathrm{sym}}^n(\pi_{f})$ is cuspidal. We now prove our main theorem.

	\subsection{Proof of the main theorem}
	
	We first recall the notation $\mathcal{I}_{k,n}(t)$ from section \ref{introduction}. 
	\begin{notation}
		For $k,n,t \in \mathbb{N}$ with $k$ even, let $\mathcal{I}_{k,n}(t)$ denote the set of all cohomological cuspidal automorphic representation $\Pi$ of ${\mathrm{GL}}_{n+1}(\A_\Q)$ corresponding to the highest weight $\mu_k$ and level $K_f^{n+1}(t)$ such that $\Pi={\mathrm{sym}^n(\pi)}$ for a cuspidal automorphic representation $\pi$ of highest weight $\big(\frac{k}{2}-1, 1-\frac{k}{2}\big)$ and level structure $K_f^{2}(t')$($t' \in \mathbb{N}$) of ${\mathrm{GL}}_{2}(\A_\Q)$.
	\end{notation}
	Note that $\mathcal{I}_{k,n}(t)\subseteq {\mathrm{Coh}}({\rm GL}_{n+1}, \mu_k, K^{n+1}_f(t))$. Furthermore, $\mathcal{I}_{k,n}(t)$ is the same as $\mathcal{D}_k(t)$ in \cite{ambi} and $E_k(t)$ in \cite{chandrasheel-sudipa} when $n$ is $2$ and $3$ respectively. Using the dimension formulae of Lemma \ref{dimensioncuspforms}, for large enough $p$, we have $|\mathcal{I}_{k,n}(p^j)|\ll_k p^{2j}$. Now we prove our main theorem which we restate here again. 
	
	\begin{theorem}
			Let $k,n \in \mathbb{N}$ with $k\geq2$ even, n$\geq 2$. We have
		$$ p^{\frac{2}{n+2}j} \ll_k |\mathcal{I}_{k,n}(p^j)| ~~~~~~~~~~ \text{ as the prime } p \to \infty$$
		where the implied constant depends upon $k$.
	\end{theorem}
	
	\begin{proof}
		Let $\pi=\bigotimes \limits_{v\leq \infty}\pi_v$ be a cuspidal automorphic representation of ${\rm GL}_2(\mathbb{A}_\mathbb{Q})$ corresponding to the level $K_f^2(p^i)$ for some $i\geq 1, p\geq 3$ prime. Then either ${\mathrm{sym}}^n(\pi)\in \mathcal{I}_{k,n}(p^j)$ for some $j\in \mathbb{N}$ or $\pi$ is CM (i.e. obtained by automorphic induction of an id\`ele class character of an imaginary quadratic extension of $\Q$). Now, $\pi_p$ corresponds to a cusp form in $S_k(\Gamma_1(p^l))$ and the correspondence is bijective if we restrict it to $S_k^{\text{new}}(\Gamma_1(p^l))$ \cite[Th. 5.19]{gelbert} where $1\leq l\leq i$. Note that $c(\pi_p)=l$ \cite[Lemma 5.16]{gelbert}. 
		
		Using Prop. \ref{conductor relation} we have that a newform in $S_k^{\text{new}}(\Gamma_1(p^l))$ corresponds to an element in $\mathcal{I}_{k,n}(p^{l'})$ only if $1\leq l' \leq (n+2)l$. Furthermore, $1\leq l\leq i \implies 1\leq l' \leq (n+2)i$. Hence we conclude that if $f \in \bigoplus_{1\leq l \leq i}S_k^{\text{new}}(\Gamma_1(p^l)) \backslash C_k(p^i)$, then ${\mathrm{sym}}^n(\pi_f)\in \mathcal{I}_{k,n}(p^{(n+2)i}) \backslash \mathcal{I}_{k,n}(p)$ where $\pi_f$ is the corresponding cuspidal representation attached to $f$. 
		As $1-\frac{1}{p^2}> \frac{3}{4}$ for $p\geq 3$, using Lemma \ref{dimensionnewforms}, we have
		\begin{eqnarray}
			\sum_{1 \leq l \leq i} \mathrm{dim}_{\mathbb{C}}S_k^{\text{new}}(\Gamma_1(p^l)) \gg_k p^{2i} . 
		\end{eqnarray}
		
		Since $|C_k(p^i)|\ll_{\epsilon} p^{i+\frac{1}{2}+\epsilon}$ as $p \to \infty$ (Prop. \ref{improvment}) which is negligible compared to $p^{2i}$, we obtain that
		\begin{eqnarray}
			p^{2i} \ll_k |\mathcal{I}_{k,n}(p^{(n+2)i})|~~~~~~~~~~~~~~~~~ \text{ as } p\to \infty.
		\end{eqnarray} 
		
		Taking $j=(n+2)i$, we deduce that $$p^{\frac{2}{n+2}j} \ll_k |\mathcal{I}_{k,n}(p^j)| ~~~\text{ as the prime } p\to \infty.$$
		
	\end{proof}
		
		
 
     \begin{remark} \label{improve}
     	Our result improves the lower bounds obtained in the previous works for $n=2, 3$.
     	When $n=2$, the main theorem implies that 
     	$$p^{\frac{j}{2}} \ll_k |\mathcal{I}_{k,2}(p^j)| \text{ as } p\to \infty,$$
     	improving the result obtained in \cite[Th. 1.5]{ambi} where the lower bound is $p^{\frac{2j}{5}}$. 
  		For $n=3$, assuming the central character of $\pi_p$ is trivial, using \cite[Prop. 12]{chandrasheel-sudipa} we obtain that the upper bound of $c({\rm sym}^3(\pi_p))$ is $2 c(\pi_{p})$. Without this assumption, i.e., for nontrivial central character we obtain $5 c(\pi_{p})$ as the upper bound, see Prop. \ref{conductor relation}. As a result, our lower bound is $p^{\frac{2j}{5}}$, improving the result of \cite{chandrasheel-sudipa} for $n=3$. Thus we have obtained better lower bounds in this article compared to the previous works for $n=2, 3$.

     \end{remark}
	
	  \noindent{\it Acknowledgement}: The first author acknowledges NBHM postdoctoral fellowship at ISI Kolkata. The second author is supported by the institute postdoctoral fellowship at HRI Prayagraj. The authors were greatly benefitted from several discussions and encouragement of Dr. Aprameyo Pal during the preparation of this article.

\end{document}